\documentclass{amsart}
\usepackage{amsmath,amssymb}
\newtheorem{theorem}{Theorem}[section]
\newtheorem{lemma}[theorem]{Lemma}

\newtheorem{proposition}[theorem]{Proposition}

\newtheorem{corollary}[theorem]{Corollary}
\newtheorem{example}[theorem]{Example}

\def\en{\mathbb N}
\def\er{\mathbb R}
\def\zet{\mathbb Z}

\def\K{\mathcal K}
\def\H{\mathcal H}

\def\R{\mathcal R}

\newcommand{\graph}{\operatorname{graph}}
\newcommand{\card}{\operatorname{card}}
\newcommand{\co}{\operatorname{co}}

\newcommand{\diam}{\operatorname{diam}}
\newcommand{\dist}{\operatorname{dist}}

\newcommand{\ep}{\varepsilon}
\newcommand{\tn}{\operatorname{Tan}}
\newcommand{\id}{\operatorname{Id}}
\newcommand{\aff}{\operatorname{Aff}}
\newcommand{\eps}{\varepsilon}

\def\eqn#1$$#2$${\begin{equation}\label#1#2\end{equation}}

\renewcommand{\labelenumi}{(\alph{enumi})}

\begin{document}

\title{On critical values of self-similar sets}
\author{Du\v san Pokorn\'y}\thanks{The author was supported by the Czech Science Foundation, project No. 201/10/J039}

\begin{abstract}
In this paper we study properties of the set of critical points for self-similar sets. We introduce simple condition that implies at most countably many critical values and we construct a self-similar set with uncountable set of critical values.
\end{abstract}
\maketitle
\section{Introduction and motivation}
Let $\phi_{1},...,\phi_{k}$ be contracting similarities on $\er^{n},$ i.e. $\phi_{i}(x)=r_{i}A_{i}x+v_{i},$ where $v_{i}\in\er^{n},$ $r_{i}\in(0,1)$ and $A_{i}$ is isometry for every $i=1,...,k$.
It is well known that by Banach fixed point theorem there is a unique compact set $K$ satisfying
$$
K=\bigcup_{i=1}^{\infty} \phi_{i}(F).
$$
Sets of such form are called self-similar.
We say that $K$ satisfies the open set condition (OSC) if there is an open set $O$ such that for every $i,j=1,..,k$ we have $f_{i}(O)\subset O$ and $f_{i}(O)\cap f_{j}(O)=\emptyset$ provided $i\not=j$.
If $A\subset \er^{n}$ is a closed set we define $N(x)=\{a\in A: \dist(x,A)=|x-a|\}$ and say that $x$ is a critical point if $x\in\co(N(x))$, $\dist(x,A)$ is then called a critical value (of $A$).
The notion of critical value of a set is often used in the theory of curvature measures.

In \cite{FR} S.~Ferry proved that in dimensions $2$ and $3$ the set of critical values has always zero Lebesgue measure 
and also constructed a set in $\er^{2}$ whose set of critical values is uncountable and a set in $\er^{4}$ whose set of critical values contains an interval. 
Later, in \cite{FU} J.H.G.~Fu gave some better estimates on the size of the set of critical values in the terms of Hausdorff dimension.

Recently, fractal counterpart of curvature measures has been introduced and studied for self-similar sets (see \cite{W},\cite{WZ}).
The way how to define such objects is the following: In the first step we approximate a self-similar set $K$ by its parallel set 
$K_{\varepsilon}=\{x:\dist(x,K)\leq\varepsilon\}$ for some $\varepsilon>0$.
If $\overline{(K_{\varepsilon})^{c}}$, the closure of complement of $K_{\varepsilon}$, has positive reach, which is true for example when $\varepsilon$ is not a critical value of $K$,
we can define (for $\overline{(K_{\varepsilon})^{c}}$) curvature measures in the sense of Federer (see \cite{F}).
In the next step, if the curvature measures of $\overline{(K_{\varepsilon})^{c}}$ exist for almost every $\varepsilon$, the fractal curvature measure of $K$ is produced by a limiting procedure with suitable scaling.

It is conjectured, that for a self-similar set satisfying open set condition the set of critical values has Lebesgue measure 0. 
So far it was even unknown whether the set of critical points is not actually always countable.
In this paper, we introduce a simple condition that implies at most countably many critical values and construct a self-similar set with uncountable set of critical values.

\section{Notation}

We will use $B(z,r)$ for closed ball in $\er^{d}$ with center $z$ and radius $r$. 

The system of all compact sets in $\er^{d}$ will be denoted $\K(\er^{d})$ and $\dist_{\H}$ will denote the Hausdorff distance on $\K(\er^{d})$.
For $L,M\in\K$ we denote $\dist(L,M)=\min_{x\in L,y\in M}|x-y|$ and define $\R(L,M)$ as the system of all closed balls of diameter $\dist(L,M)$ that intersect both $L$ and $M.$

Let $M\subset \zet$ be finite. For $I\in M^{<\omega}$ we will denote $|I|$
the length of $I,$ and $\prec$ will be used for classical lexicographic ordering.

For $I\in M^{<\omega}$ or $I\in M^{\omega}$ and $n\in \en$ denote $I(n)$ the $n$-th coordinate of $I$ and 
define $I|_{n}\in M^{n}$ as $I|_{n}(i)=I(i)$ for $i=1,..,n.$

For $I,J\in M^{<\omega}$ define $I^*J\in M^{|I|+|J|}$ as $I^*J(i)=I(i)$ for $i=1,..,|I|$ and $I^*J(|I|+i)=J(i)$ for $i=1,..,|J|.$

We will write $I\triangleleft J$ if there is an $n\in \en$ such that $I=J|_{n}.$

If $K$ is a self-similar set with system of generating similarities $\phi_{1},...,\phi_{k}$ $I\in \{1,...,k\}^{<\omega}$ with similarity ratios $r_{1},...,r_{k},$ 
we denote $\phi_{I}=\phi_{I(1)}\circ ... \circ \phi_{I(|I|)}$, $r_{I}=r_{I(1)}\cdot ... \cdot r_{I(|I|)}$ and $K_{I}=\phi_{I}(K).$

For $a,b,c\in\er^{2}$ we will denote by $\angle(a,b,c)\in[0,\pi]$ the angle between vectors $a-b$ and $c-b.$

\section{A simple condition that implies countably many critical values}

\begin{theorem}\label{count}
Let $K\subset\er^{n}$ be a self-similar set such that $\co K$ is a polytope.
Then $K$ has at most countably many critical values. 
\end{theorem}
\begin{proof}
Follows directly from Lemma~\ref{L1} and Proposition~\ref{prop} below.
\end{proof}

\begin{corollary}
Let $K\subset\er^{n}$ be a self-similar set generated by the mappings $\phi_{i}(x)=r_{i}A_{i}x+v_{i},$ $i=1,...,k.$  
If for every $i$ there is an $n_{i}$ such that $A_{i}^{n_{i}}=\id$ then $K$ has at most countably many critical values. 
\end{corollary}
\begin{proof}
Follows Theorem 1. in \cite{P} which states, that under the assumptions of the corollary $\co K$ is a polytope and therefore we can use Theorem~\ref{count}.
\end{proof}

\begin{lemma}\label{polycond}
Let $k,m\in\en$ and let $K_{j}^{i},$ $i=1,...,k$, $j=1,...,m$ be compact sets such that $\co K_{j}^{i}$ is always a polytope.
Let $\Gamma :\{1,...,k\}^{<\omega}\to \K(\er ^{d}),$ $c>0,$ $1>q>0,$ similarities $\phi^{I}:\er^{d}\to\er^{d}$ and numbers $j(I)=1,...,m$, $I\in\{1,...,k\}^{<\omega}$ be given
such that for every $I\in\{1,...,k\}^{<\omega}$ the following conditions hold:
\begin{enumerate}
\item $\Gamma(I)=\bigcup _{i=1,...,k}\Gamma(I^*{i})$
\item $\Gamma(I^*i)=\phi^{I}(K_{j(I)}^{i})$ for every $i=1,...,k$
\item $\diam \Gamma(I)\leq cq^{|I|}$.
\end{enumerate}
Then there is $\gamma>0$ such that for every $I\in\{1,...,k\}^{<\omega}$ we have
\begin{equation}\label{distdist}
\dist(x,\partial \co \Gamma(I))\geq \gamma\dist(x,\Gamma(I)\cap\partial \co \Gamma(I))
\end{equation}
for every $x\in\Gamma_{I}.$
\end{lemma}

\begin{proof}
Let $F^{I}$ be system of all faces of $\co\Gamma(I).$
For every $f\in F^{I}$ and $i=1,...,k$ satisfying $\Gamma(I^*i)\cap f=\emptyset$ define 
$$
\gamma_{f,I,i}=\frac{\dist(\Gamma(I^*i),\aff f)}{\max_{x\in \Gamma(I^*i)}\dist(x,\Gamma(I)\cap \aff f)}>0 \quad \text{and}\quad \tilde\gamma=\min_{i,I,f}\gamma_{f,I,i}>0.
$$
For every $f\in F^{I}$ and $i=1,...,k$ satisfying $\Gamma(I^*i)\cap f\not=\emptyset$ and $\dim\aff(\co\Gamma(I^*i)\cap f)<\dim\aff f$ define
$$
\beta_{f,I,i}=\min_{x\in\Gamma(I^*i)}\frac{\dist(x,\aff f)}{\dist(x,\aff(\Gamma(I^*i)\cap \aff f))}>0 \quad \text{and}\quad \beta=\min_{i,I,f}\beta_{f,I,i}>0.
$$
Note that in the definition of both $\tilde\gamma$ and $\beta$ we could have use minimum and obtain their positivity,
since due to the property $(b)$ we are in fact working with finite sets of positive numbers 
(numbers $\beta_{f,I,i}$ are positive because the convex hull of every $\co K_{j}^{i}$ is a polytope). 
Put $\gamma=\beta^{d}\tilde\gamma.$
It remains to prove that (\ref{distdist}) holds.

Choose $I\in\{1,...,k\}^{<\omega}$ and $x\in \Gamma(I).$
If $x\in\partial\co\Gamma(I)$ then the statement is trivial so we can suppose $x\in(\co\Gamma(I))^\circ.$
First we will prove that for every $f\in F_{I}$ we have 
\begin{equation}\label{linf}
\dist(x,\aff f)\geq \gamma\dist(x,\Gamma(I)\cap\aff f).
\end{equation}

By properties $(a)$ and $(c)$ there is $J\in\{1,...,k\}^{<\omega}$ such that 
$x\in\Gamma(I^*J),$ $f\cap \Gamma(I^*J)\not=\emptyset$ and for every $i\in\{1,...,k\}$ we have either
$x\not\in\Gamma(I^*J^*i)$ or $f\cap \Gamma(I^*J^*i)=\emptyset$.
Then for every $0\leq l\leq|J|+1$ there is $f_{l}\in F_{I^*J|_{l}}$ such that $\aff f_{l}=\aff(\co\Gamma(I^*J|_{l})\cap f).$
Now, fro $0\leq l\leq|J|-1$, there are two cases:
\begin{enumerate}
\item $\dim\aff f_{l}=\dim\aff f_{l+1}$ which implies $\dist(x,\aff f_{l})\geq \dist(x,\aff f_{l+1})$
\item $\dim\aff f_{l}>\dim\aff f_{l+1}$ which implies $\dist(x,\aff f_{l})\geq \beta\dist(x,\aff f_{l+1}).$
\end{enumerate}
Moreover, $\dist(x,\aff f_{|J|})\geq \tilde\gamma\dist(x,\Gamma(I^*J)\cap\aff f_{|J|+1}).$
These estimates together with the fact that case $(b)$ can occur at most $d$ times gives us \ref{linf}.
Finally, by (\ref{linf}) we obtain for every $x\in\Gamma_{I}$
$$
\begin{aligned}
\dist(x,\partial \co \Gamma(I))=&\min_{f\in F_{I}}\dist(x,\aff f)\geq \gamma \min_{f\in F_{I}}\dist(x,\Gamma(I)\cap\aff f) \\
\geq&\gamma\dist(x,\Gamma(I)\cap\partial \co \Gamma(I)).
\end{aligned}
$$
\end{proof}

\begin{lemma}\label{L1}
Let $K$ be a self-similar set satisfying assumptions of Theorem~\ref{count}.
Then for every $S\subset \er^{n}\setminus K$ open strictly convex and with smooth boundary, the set $K\cap \overline S$ is finite.
\end{lemma}
\begin{proof}
Choose $S\subset \er^{n}\setminus K$ open strictly convex and with smooth boundary.
Under our assumptions $\co K$ (and by self similarity $\co K_{I}$ for every $I\in\{1,...,k\}^{<\omega}$) is a polytope (see \cite{P}, Theorem 1.).
So if we consider the mapping $\Gamma:I\mapsto K_{I}$, it satisfies the conditions of Lemma~\ref{polycond}.
Therefore there is a constant $\gamma$ such that
\begin{equation}\label{1gamma}
\dist(x,\partial \co K)\geq \gamma\dist(x,K\cap\partial \co K)
\end{equation}
for every $x\in K.$
Now, since $\partial S$ is smooth we can find $\delta_{1}>0$ such that if for some $I\in\{1,...,k\}^{<\omega}$ we have $\diam K_{I}\leq\delta_{1}$ then for every $x\in S\cap\co K_{I}$
\begin{equation}\label{2gamma}
\dist(x,S\cap\partial\co K_{I})\leq \frac{1}{2}\gamma\dist(x,(\partial\co K_{I})\cap K_{I})).
\end{equation}
Let 
$$
\Sigma(\delta)=\{I\in\{1,...,k\}^{<\omega}:r_{I}\diam K\leq\delta<r_{I|_{|I|-1}}\diam K\}.
$$
Then $\Sigma(\delta_{1})$ is finite, $K=\bigcup_{I\in\Sigma(\delta_{1})}K_{I}$ and by $(\ref{1gamma})$ and $(\ref{2gamma})$ 
we have $\overline S\cap K_{I}\subset \partial\co K_{I}$ for every $I\in\Sigma(\delta_{1})$.
This means that every $x\in \overline S\cap K_{I}$ belongs to some $n-1$-dimensional face of $K_{I}$.
Now consider all sets and linear spaces of the form $K_{J}\cap \aff f$ and $M=\aff f$, respectively, for $J\in \Sigma(\delta_{1})$ and $f$ face of $K_{J}$.
and open strictly convex sets and with smooth boundary of the form $S\cap M.$
Since there is again at most finitely many of them we can finitely many times apply Lemma~\ref{polycond}.
This way we obtain $\delta_{2}>0$ with the property that for every $I\in\Sigma(\delta_{2})$ we have that every $x\in \overline S\cap K_{I}$ belongs to some $n-2$-dimensional face of $K_{I}$  
Finally, after $n$ steps we can obtain $\delta_{n}$ with the property that for every $I\in\Sigma(\delta_{n})$ the set $\overline S\cap K_{I}$ is empty or singleton.
This means that $\card S\cap K\leq\card \Sigma(\delta_{n})<\infty.$
\end{proof}

We will use the notation $g_0(t)=\sqrt{r_0^2-t^2}$, $0\leq t\leq r_0$.
 
\begin{lemma}  \label{L2}
Let $r_0>0$ be given and denote
$$
U=\{ (t,y):\, 0<t<r_0,\, g_0(t)\leq y<r_0\},
$$
where $g_0(t)=\sqrt{r_0^2-t^2}$ for $0\leq t\leq r_0$.
Let $S$ be a relatively closed subset of $U$ such that $(0,r_0)\in\overline{S}$. Then there exists a continuously differentiable concave function $h$ on $(0,r_0)$ such that
\begin{eqnarray}
&&\forall t\in (0,r_0):\, (t,y)\in S\implies h(t)\leq y, \label{Pr1}\\
&&(0,r_0)\in\overline{S\cap\graph h}. \label{Pr2}
\end{eqnarray}
\end{lemma}

\begin{proof}
If $(0,r_0)\in\overline{S\cap\graph g_0}$ we can take $g_0$ for $h$ and we are done. Assume thus that there exists $t_0>0$ such that
\begin{equation}  \label{empty}
([0,t_0]\times\er)\cap S\cap\graph g_0=\emptyset .
\end{equation}
Choose $z\in (0,t_{0}]$ and define
$$
h_{z}(t):=\left\{
\begin{array}{ll}
\min\{\frac{r_{0}^{2}-tz}{\sqrt{r_0^2-t^2}},r_{0}\},\quad &0<t\leq t_0,\\
g_0(t),&t_0\leq t<r_0
\end{array} \right. .
$$
By our assumptions there is $(t,y)\in S$ with $0<t<z$ such that $h_{z}(t)>y.$
Denote by $C_{z}$ the system of all concave and continuously differentiable functions $f:[0,r_{0})\mapsto(0,r_{0})$ 
such that there is $\varepsilon_{f}>0$ with $\{x:f(x)\not=g_{0}(x)\}\subset[\varepsilon_{f},z].$
Since $h_{z}$ can be (uniformly) approximated by functions from $C_{z}$ we can find $f\in C_{z}$ such that $f(t)<y.$
Put $S_{z}=\{(u,v)\in S:f(u)\leq v\}$ and
$$
\alpha=\inf\{\beta\in[0,1]:\;\text{there is}\:\: (u,v)\in S_{z}\;\;\text{with}\;\; \beta f(u)+(1-\beta)g_{0}(u)\geq v\}.
$$

Put $g_{z}(u)=\alpha f(u)+(1-\alpha)g_{0}(u).$
Then there is $u\in[\varepsilon_{f},z]$ such that $(u,g_{z}(u))\in S$, if $(c,d)\in S$ then $g_{z}(c)\leq d$ and $f$ is continuously differentiable concave function such that 
$g_{z}=g_{0}$ on $[0,r_{0})\setminus[\varepsilon_{f},z]$ and $g_{z}\geq g_{0}$ on $[0,r_{0}).$

Now use the above procedure inductively to construct for every $n\in\en$ continuously differentiable convex functions $h_{n}:[0,r_{0})\mapsto(0,r_{0})$ and intervals $[a_{n},b_{n}]$ such that 

\begin{itemize}
\item $a_{n}\to 0,$ 
\item $a_{n}>b_{n+1}$
\item there is $u_{n}\in[a_{n},b_{n}]$ such that $(u_{n},h_{n}(u_{n}))\in S$
\item if $(c,d)\in S$ then $h_{n}(c)\leq d$ 
\item $h_{n}=g_{0}$ on $[0,r_{0})\setminus[a_{n},b_{n}]$
\item $h_{n}\geq g_{0}$ on $[0,r_{0}).$
\end{itemize}
To finish the proof it suffices to put $h=\max_{n\in\en}h_{n}.$
\end{proof}

\begin{proposition}\label{prop}
Let $K\subset\er^{n}$ be a set with the property that for every $S\subset \er^{n}\setminus K$ open strictly convex with smooth boundary the set $K\cap \overline S$ is finite.
Let $r>0$ be a critical value of $K$. Then there exists $\ep>0$ such that there are no critical values of $K$ in $(r-\ep,r)$.
\end{proposition}

\begin{proof}
Assume, for the contrary, that there exists a sequence $r_i\nearrow r_0$ of critical values of $K$, and let $s_i$ be the corresponding critical points ($i=0,1,\ldots$). Passing to a subsequence, we can assume that $s_i\to s_0$ and $\frac{s_i-s_0}{|s_i-s_0|}\to u\in S^{d-1}$ (clearly, $s_i\neq s_0$ for $i\geq 1$ since the critical values are different). 
We know from the definition of criticality that for all $i\geq 0$,
\begin{equation}  \label{E1}
B_{r_i}(s_i)^{\circ}\cap F=\emptyset
\end{equation}
and
\begin{equation}  \label{E2}
s_i\in \co (B_{r_i}(s_i)\cap K).
\end{equation}
Denote
$$g(t):=\dist (s_0+tu,(s_0+tu+u^\perp)\cap K).$$
Using \eqref{E1} for $i=0$, we get that 
\begin{equation} \label{E3}
g(t)\geq \sqrt{r_0^2-t^2},\quad 0<t<r_0.
\end{equation}
Denote, further, $T:=\{ t\in(0,r_0):\, g(t)<r_0\}$ and $S:=\{(t,g(t)):\, t\in T\}$. It follows from \eqref{E2} that $0\in\overline{T}$ and, hence, $(0,r_0)\in\overline{S}$. We can thus apply Lemma~\ref{L2} and obtain a concave function $h$ defined on an interval $(0,t_0)$ with $h\leq g$ and such that $(0,r_0)\in\overline{S\cap\graph h}$. The set
$$C:= B_{r_0}(s_0)^{\circ}\cup\{ s_0+tu+v:\, |v|<h(t),\, t\in (0,t_0)\}$$
does not intersect $K$ due to \eqref{E1} and since $h\leq g$.
On the other hand, $C$ is an open convex set with smooth boundary that intersects $K$ in infinitely many points, which is a contradiction with our assumptions.
\end{proof}

If we for some $i=1,...,k$ allow $\phi_{i}(x)=r_{i}A_{i}x+v_{i},$ with $A_{i}^{n}\not =Id$ for every $n\in\en,$ we can obtain $K\cap\partial S$ infinite for some $S$ convex with smooth boundary such that $S^\circ\cap K=\emptyset.$ Note that due to Proposition~\ref{prop} every self-similar set with uncountable set of critical values provides such an example. 
But we make it with much less work and obtain even stronger result.

\begin{example}
Define $\phi_{0}, \phi_{1}:\er^{2}\mapsto\er^{2}$ by formulas 
$$
\phi_{0}(x,y)=\frac{1}{5}
\begin{pmatrix}
\cos\alpha & -\sin\alpha\\
	\sin\alpha & \cos\alpha
\end{pmatrix}
\cdot
\begin{pmatrix}
	x\\
	y
\end{pmatrix}
+
\begin{pmatrix}
	1\\
	0
\end{pmatrix}
$$

$$
\phi_{1}(x,y)=\frac{1}{5}
\begin{pmatrix}
\cos\alpha & -\sin\alpha\\
	\sin\alpha & \cos\alpha
\end{pmatrix}
\cdot
\begin{pmatrix}
	x\\
	y
\end{pmatrix}
-
\begin{pmatrix}
	1\\
	0
\end{pmatrix}
$$for some irrational angle $\alpha$ and define $K$ as the unique compact set that satisfies $K=\phi_{0}(K)\cup\phi_{1}(K).$

Then there is a strictly convex compact set $U$ with smooth boundary such that $U^{\circ}\cap K=\emptyset$ and $U\cap K$ is uncountable.
\end{example}

If we consider any direction $l_{k}=(\cos(k\alpha),\sin(k\alpha))$ for $k\in \en_{0}$ then there are exactly two (non degenerated) maximal line segments $l^{0}_{k},l^{1}_{k}\subset \partial(\co K)$  parallel to $l_{k}.$
Denote system of all line segments $l^{0}_{k}$ and $l^{1}_{k}$ by $L.$
Moreover, $\overline{\cup L}=\partial (\co C)$ and since directions $l_{k},$ $k\in \en_{0}$ are dense in $S^{1}$ the set $\partial (\co K)$ is a continuously differentiable Jordan curve.

Define $\textbf{K}$ as a system of all sets of the form $K_{I}:=\phi_{I}(K)$ for some $I\in \{0,1\}^{<\omega}$ 
and $L_{I}$ as a system of all line segments of the form $\phi_{I}(l)$ for some $l\in L.$
Note that every $l\in L_{I}$ belongs to $\partial (\co {K_{I}})$ and that $\overline{\cup L_{I}} =\partial (\co {K_{I}})$.

Let $\varepsilon>0.$ We say that a strictly convex compact set with smooth boundary $A\subset \er^{2}$ cuts $\varepsilon$-well $M=K_{I}\in \textbf{K}$ if $ A\cap\partial  M=[a,b]\in L_{I}$ and
$|\tn(\partial A,a)-\tn(\partial A,b)|\leq\varepsilon$.
For $a,b,v\in \er^{2}$ denote $P(a,b,v)$ the (closed) half plane with $[a,b]\subset\partial P(a,b,v)$ that contains $v.$

\begin{lemma}\label{l}
Let $M=K_{I}\in \textbf{K}$ and $\varepsilon>0.$ Let $A\subset \er^{2}$ be a strictly convex compact set with smooth boundary that cuts well $M.$
Then there is a strictly convex compact set with smooth boundary $\tilde A$ and $M^*, M'\in \textbf{K}$ such that  
\renewcommand{\labelenumi}{(\arabic{enumi})}
\begin{enumerate}
\item $M^*,M'\subset M$ and $M^*\cap M'=\emptyset$
\item $\tilde A\subset A$
\item $A\setminus \tilde A\subset (\co(M)\cap A)+ B(0,\varepsilon)$
\item $\tilde A$ cuts $\varepsilon$-well both $M^*$ and $M'$
\item $\diam M^*\cap \tilde A,\diam M'\cap \tilde A<\varepsilon.$
\item $\mathcal{H}^{1}(\partial \tilde A\setminus \partial A)<\diam(\co(M)\cap A)+\varepsilon$
\item $C\cap \tilde A\subset \partial\tilde A +B(0,\varepsilon)$
\end{enumerate} 
\end{lemma}
\begin{proof}
Due to the self-similarity of $C$ we can assume $M=C$ and $ A\cap\partial  M=[a,b]=l^{0}_{0}.$
Put $M^*=K_0$ and $M'=K_1.$
We can suppose that $a\in M^*$ and $b\in M'$.
Since $\partial M^*$ and $\partial M'$ are both continuously differentiable Jordan curves with $a-b\in \tn(\partial M^*,a)$ and $a-b\in (\partial M',b)$
and since $\overline{\cup L_0}=\partial M^*$ and
$\overline{\cup L_1}=\partial M'$ we can choose $[a_{0},b_{0}]\in L_0$ and $[a_{1},b_{1}]\in L_1$
satisfying:
\begin{itemize}
\item $a_{0},b_{0}\subset B(a,\frac{\varepsilon}{2}),a_{1},b_{1}\subset B(b,\frac{\varepsilon}{2})$
\item $A\cap P(a_{0},b_{0}, a),A\cap P(a_{1},b_{1}, b)\subset (M\cap A)+ B(0,\frac{\varepsilon}{2})$
\item $[a_{0},b_{0}]\subset (P(a_{1},b_{1},a)\cap A)^{\circ}$ and $[a_{1},b_{1}]\subset (P(a_{0},b_{0},b)\cap A)^{\circ}$
\end{itemize} 
It is easy to see that these conditions are sufficient for existence of $\tilde A$ satisfying properties 2.-7.
Property 1. holds due to the choice of $M^*$ and $M'.$
\end{proof}

\begin{lemma}\label{k}
There are strictly convex compact sets $A_{n}$ with smooth boundary and systems $\{M_{i}^{n}\}_{i=1}^{2^{n}}\subset \textbf{K}$ such that for every $n\in \en_{0}$ the following conditions hold:
\renewcommand{\labelenumi}{(\arabic{enumi})}
\begin{enumerate}
\item $M_{i}^{n}\cap M_{j}^{n}=\emptyset$ for $i\not=j$
\item if $n>0$ then $A_{n}\subset A_{n-1}$
\item $A_{n}$ cuts $2^{-n}$-well $M_{i}^{n}$ for $i=1..2^{n}$
\item $\diam (A_{n}\cap M_{i}^{n})<2^{-(2n+2)}$
\item if $n>0$ then $M_{2j-1}^{n}\cup M_{2j}^{n}\subset M_{j}^{n-1}$ for $j=1..2^{n-1}$
\item if $n>0$ $\mathcal{H}^{1}(\partial A_{n}\setminus \partial A_{n-1})<2^{-n}$.
\item $K\cap A_{n}\subset \partial A_{n} +B(0,2^{-n})$
\end{enumerate}
\end{lemma}
\begin{proof}
We will proceed by induction by $n.$

$n=0:$ Put $M_{1}^{0}=K$ Choose $A_{0}$ as an arbitrary ball of diameter smaller than $1$ that cuts well $K$.
Validity of all conditions $(1)-(7)$ is obvious.

\textit{Induction step:} Suppose that we have $A_{n-1}$ and $\{M_{i}^{n-1}\}_{i=1}^{2^{n-1}}.$ 
Use Lemma~\ref{l} separately for every $M=\{M_{i}^{n}\},$ $i=1..2^{n-1}$ with $A=A_{n-1}$ and 
$$
\varepsilon= \frac{1}{4}\min (2^{-2n},\min_{i\not= j} \dist (A_{n-1}\cap M_{i}^{n-1}, A_{n-1}\cap M_{j}^{n-1} ))
$$
to obtain $2^{n-1}$ strictly convex compact sets (denote them $\tilde A_{j}$ for $j=1..2^{n-1}$) 
and corresponding $2^{n}$ sets from $\textbf{K}$ (denote them $M_{i}^{n}$ for $i=1..2^{n}$ in a way that $M_{2j-1}^{n}$ and $M_{2j}^{n}$ correspond to $\tilde A_{j}$ for $j=1..2^{n-1}$).
Put $A_{n}=\cap_{j}\tilde A_{j}.$
$A_{n}$ is strictly convex and compact since all $\tilde A_{j}$ are.
Condition $(1)$ holds due to the induction hypothesis and property $(1)$ from Lemma~\ref{l} 
and condition $(2)$ due to property $(2)$ from Lemma~\ref{l} .
Condition $(3)$ holds due to properties $(3)$ and $(4)$ from Lemma~\ref{l} and the fact that $\varepsilon\leq\frac{1}{4}\min_{i\not= j} \dist (A_{n-1}\cap M_{i}^{n-1}, A_{n-1}\cap M_{j}^{n-1} ).$
Condition $(4)$ holds by property $(5)$ from Lemma~\ref{l} and the fact that $\varepsilon\leq2^{-(n+2)}.$
Condition $(5)$ holds due to property $(1)$ from Lemma~\ref{l}.
To prove condition $(6)$ write
$$
\begin{aligned}
\mathcal{H}^{1}(\partial A_{n}\setminus \partial A_{n-1})= \mathcal{H}^{1}(\partial (&\bigcup _{j=1}^{2^{n-1}}\tilde A_{j})\setminus \partial A_{n-1})=\sum _{j=1}^{2^{n-1}}\mathcal{H}^{1}(\partial \tilde A_{j}\setminus \partial A_{n-1})\\
&< 2^{n-1}(2^{-2n}+2^{-2n-2})<2^{-n}, 
\end{aligned}
$$
where the second equality holds due to property $(3)$ from Lemma~\ref{l} and the fact that $\varepsilon\leq\frac{1}{4}\min_{i\not= j} \dist (A_{n-1}\cap M_{i}^{n-1}, A_{n-1}\cap M_{j}^{n-1} )$
and the first inequality holds by property $(6)$ from Lemma~\ref{l}, induction hypothesis (property $(4)$) and the fact that  $\varepsilon\leq2^{-(n+2)}.$
\end{proof}

\begin{theorem}
There is a strictly convex set $U$ with smooth boundary such that $U^{\circ}\cap K=\emptyset$ and $U\cap K$ is uncountable.
\end{theorem}
\begin{proof}
Consider sets $A_{n}$ and systems $\{M_{i}^{n}\}_{i=1}^{2^{n}}$ form Lemma~\ref{k}. 
Put $U=\cap_{n} A_{n}$. As an intersection of compact sets $U$ is compact. 
To prove that $U$ is strictly convex suppose for contradiction that there is a line segment $[a,b]\subset \partial U.$ 
Find $k\in \en$ such that $\sum _{n=k}^{\infty}2^{-n}<|a-b|.$
This by property $(6)$ from Lemma~\ref{k} implies that there is a non degenerated line segment $[c,d]\subset [a,b]\cap \partial A_{k-1}$ which is in contradiction with the fact that $A_{k-1}$ is strictly convex. Property $(3)$ smoothness of boundary of all $A_{n}$ implies that $A$ has smooth boundary as well.

For $I\in \{0,1\}^{\omega}$ define $x_{I}$ by $x_{I}=\bigcap_{n=1..\infty}M^{n}_{I(n)}.$ 
Due to properties $(1),(4)$ and $(5)$ from Lemma~\ref{k} is the definition correct and for $I,J\in \{0,1\}^{\omega},$ $I\not=J,$ we have $x_{I}\not= x_{J}.$
Put $X=\bigcup_{I\in \{0,1\}^{\omega}}x_{I}.$
Due to property $(3)$ from Lemma~\ref{k} we have $X\subset A_{n}$ for every $n\in \en_{0}$
and by property $(2)$ from Lemma~\ref{k} and the compactness of all sets $A_{n}$ and $K$ we have $x_{I}\in U$ for every $I\in \{0,1\}^{\omega}$ and so $X$ is an uncountable set in $K\cap U.$
Finally, by property $(7)$ from Lemma~\ref{k} we have $U^{\circ}\cap K=\emptyset.$
\end{proof}

\section{Set with uncountably many critical values}
An example of a set with uncountable many critical values can be produced similarly to the previous example, 
but this time we would have to choose $\alpha$ more carefully by manoeuvring with copies in certain levels.
Most of the technicalities with this approach would be because of the fact that the directions between the different copies on our set would depend on $\alpha.$
To avoid these problems we use a small trick that we add to the previous construction two much smaller copies near the center of the set which will be symmetrical by the $x$-axis.

Put 
$$
A^{\alpha}=
\begin{pmatrix}
\cos\alpha & -\sin\alpha\\
	\sin\alpha & \cos\alpha
\end{pmatrix},\;
e_1=
\begin{pmatrix}
	1\\
	0
\end{pmatrix}
\quad
\text{and}\quad
e_2=
\begin{pmatrix}
	0\\
	1
\end{pmatrix}
\quad
\text{and}\quad
q=\frac{1}{1000}
$$
and define contracting similarities $\phi_{\pm1}^{\alpha}, \phi_{\pm2}^{\alpha}:\er^{2}\mapsto\er^{2}$ by formulas 
$$
\phi_{\pm1}^{\alpha}(x)=qA^{\alpha}x\pm \frac{1}{2}e_1,
\quad\text{and}\quad
\phi_{\pm2}(x)=q(qx\pm e_2),
$$
for an angle $\alpha\in [0,2\pi]$ and define $K^{\alpha}$ as the unique compact set that satisfies $K^{\alpha}=\phi_{-1}^{\alpha}(K^{\alpha})\cup\phi_{+1}^{\alpha}(K^{\alpha})\cup\phi_{-2}^{\alpha}(K^{\alpha})\cup\phi_{+2}^{\alpha}(K^{\alpha}).$
We also define $D_{\alpha}=\diam K^{\alpha}.$
It is easy to see that $1\leq D_{\alpha}\leq 1+2q.$
For the rest of the paper we will think of interval $[0,2\pi]$ metrically as a sphere (with $0=2\pi$), all operations will be considered $\mod _{2\pi}$
and intervals will be considered to preserve classical intervals in $[0,2\pi].$

For $I=2^*J\in\{2\}\times\{\pm1,\pm2\}^{<\omega}$ and $\alpha\in[0,2\pi]$ define 
$$
M_{I}^{\alpha}=\min\{y:\text{there is $x\in\er$ such that}\;(x,y)\in K_{I}^{\alpha}\}
$$
and
$$
\R_{I}^{\alpha}=\{B((x,0),M_{I}^{\alpha}):(x,M_{I}^{\alpha})\in K_{I}^{\alpha})\}=\R(K^{\alpha}_{I},K^{\alpha}_{-2^*J}).
$$ 
Note that $M_{I}^{\alpha}$ is always a critical value of $K^{\alpha}_{I}\cup K^{\alpha}_{-2^*J}.$

We start with few simple observations.

\begin{lemma}\label{sqrt}
Let $\frac{2}{3}q\leq a\leq \frac{4}{3}q$ and $0<b\leq q^{2}.$ 
Then
$$
\frac{b}{2a}<a-\sqrt{a^{2}-b} <\frac{3b}{5a}\quad \text{and} \quad 
\frac{2b}{5a}<\sqrt{a^{2}+b}-a <\frac{b}{2a}.
$$
\end{lemma}
\begin{proof}
Simple computation.
\end{proof}

\begin{lemma}\label{pom1}
Let $K,L,M\in \K(\er^{2})$ and $d>0.$
Suppose that $\dist(L,M)\geq d$ and $\dist(K,L)\leq \dist(K,M).$
Then for every $R\in \R(K,L)$ we have
$$
\dist(R,M)\geq \sqrt{\frac{\dist(K,L)^{2}}{4}+\frac{d^{2}}{2}}-\frac{\dist(K,L)}{2}.
$$
\end{lemma}
\begin{proof}
Choose $u\in K$ and $v\in L$ such that $R=\R(u,v)$ and choose $s=(x,y)\in M.$
By choosing the appropriate system of coordinates we can suppose $u=(0,r)$ and $v=(0,-r)$ and so $R=B(0,r)$ and what we want to prove is $x^{2}+y^{2}\geq \frac{d^{2}}{2}+r^{2}.$
Now,
\begin{equation}\label{firstin}
d^{2}\leq x^{2}+(y+r)^{2}=x^{2}+y^{2}+r^{2}+2yr
\end{equation}
and
\begin{equation}\label{secondin}
4r^{2}\leq x^{2}+(y-r)^{2}=x^{2}+y^{2}+r^{2}-2yr.
\end{equation}
By adding (\ref{firstin}) to (\ref{secondin}) we get $2x^{2}+2y^{2}+2r^{2}\geq d^{2}+4r^{2},$ which implies the desired inequality.
\end{proof}

\begin{corollary}\label{cr}
Let $I\in\{2\}\times\{-1,1\}^{<\omega},$ $\alpha \in[0,2\pi]$ and $L\subset\{2\}\times\{-1,1\}^{|I|-1}.$
Suppose that 
$$
M^{\alpha}_{I}\leq \min _{J\in L} M^{\alpha}_{J}.
$$
Then for every $R\in\R^{\alpha}_{I}$ and every $J\in L,$ we have
$$
\dist(R, K^{\alpha}_{J})\geq \frac{q^{2|I|-1}}{7}.
$$

\end{corollary}
\begin{proof}
Choose $R\in\R^{\alpha}_{I}$ and $J\in L.$
Then $\dist( K^{\alpha}_{I}, K^{\alpha}_{J})\geq q^{|I|}-2D_{\alpha}q^{|I|+1}$ and
applying Lemma~\ref{pom1} (first inequality) and Lemma~\ref{sqrt} together with the fact that $\frac{2}{3}q\leq M^{\alpha}_{I}\leq \frac{4}{3}q$ (third inequality) we can estimate
$$
\begin{aligned}
\dist(R, K^{\alpha}_{J})&\geq \sqrt{\frac{(M^{\alpha}_{I})^{2}}{4}+\frac{(q^{|I|}-2D_{\alpha}q^{|I|+1})^{2}}{2}}-\frac{M^{\alpha}_{I}}{2}\\
&\geq \sqrt{\frac{(M^{\alpha}_{I})^{2}}{4}+\frac{q^{2|I|}}{4}}-\frac{M^{\alpha}_{I}}{2} \\
&\geq \frac{q^{2|I|}}{2}\cdot\frac{2}{5M^{\alpha}_{I}}\geq \frac{q^{2|I|-1}}{7}.
\end{aligned}
$$
\end{proof}

Next is the key lemma that will be essential for the inductive procedure in Proposition~\ref{key}.

\begin{lemma}\label{intervals}
Suppose that $[a,b]\subset[\frac{\pi}{2},\frac{3\pi}{2}]$ and $k_{0}\in\en$.
Then there is $k>k_{0}$ and $[c,d]\subset[a,b]$ such that for every $\alpha\in[c,d]$, every $I\in\{2\}\times\{\pm2\}^{k}$ and every $R_{\pm}\in\R^{\alpha}_{I^*\mp1}$ we have 
\begin{enumerate}
\item $\dist(K^{\alpha}_{I^*\pm1},R_{I^*\mp1})\geq q^{2|I|+2}$
\item $11q^{2(|I|+1)}\geq|M^{\alpha}_{I^*1}-M^{\alpha}_{I^*-1}|\geq q^{2(|I|+1)}$
\item $\dist_{\H}(R_{I^*1},R_{I^*-1})\leq 2q^{|I|+1}.$
\end{enumerate}
\end{lemma}

\begin{proof}
First note that the set $\{\frac{2n\pi+7q^{m+1}}{|I|+m}:n,m\in\en\}$ is dense in $[0,2\pi].$
Therefore we can find $l,k\in\en,$ $k>k_{0},$ such that $\kappa:=\frac{2l\pi+7q^{k+1}}{k}\in(a,b)$. 
This also means that $k\kappa=7q^{k+1}$ and so we can find $\eps>0$ such that $[\kappa-\eps,\kappa+\eps]\subset[a,b]$ 
and such that 
\begin{equation}\label{incl}
k\alpha\in[4q^{k+1},10q^{k+1}]\quad \text{for every}\quad \alpha\in[\kappa-\eps,\kappa+\eps].
\end{equation}
Put $[c,d]=[\kappa-\eps,\kappa+\eps]$.

Choose $J\in\{\pm1\}^{k}$ and $\alpha\in[c,d]$ and put $K^{\pm}=K_{I^*J^*\pm1}^{\alpha}$, $M^{\pm}=M^{\alpha}_{I^*J^*\pm1}$ and choose $(x_{\pm},M^{\pm})\in K^{\pm}$ and $R_{\pm}\in\R^{\alpha}_{I^*J^*\pm1}.$
Without any loss of generality we can suppose that $M^{-}<M^{+}.$
The conditions we want to prove are then
\begin{enumerate}
\item $\dist(K^{\pm},R_{\mp})\geq q^{2(|I|)+2}$
\item $10q^{|I|+1}\geq M^{-}-M^{+}\geq q^{|I|+1}$
\item $\dist_{\H}(R_{+},R_{-})\leq 2q^{|I|+1}$.
\end{enumerate}
We prove the second condition.
First observe that 
$$
K^{+}=K^{-}+q^{|I|+1}(\cos((|I|)\alpha),\sin((|I|)\alpha)).
$$
This implies 
$$
(q^{|I|+1}-2D_{\alpha}q^{|I|+2})\sin(|I|\alpha)\leq M^{+}-M^{-}\leq (q^{|I|+1}+2D_{\alpha}q^{|I|+2})\sin(|I\alpha)
$$ 
and 
$$
(q^{|I|+1}-2D_{\alpha}q^{|I|+2})\cos(|I|\alpha)\leq x^{+}-x^{-}\leq (q^{|I|+1}+2D_{\alpha}q^{|I|+2})\cos(|I|\alpha).
$$
Using the fact that $\alpha\geq\sin(\alpha)\geq\frac{1}{2}\alpha$ if $\alpha\in[0,\frac{\pi}{2}]$ and (\ref{incl}) we have
$$\label{mineq}
q^{2(|I|+1)}\leq M^{+}-M^{-}\leq 11q^{2(|I|+1)}.
$$
which proves the second condition.

The third condition follows similarly from
$$
\dist_{\H}(R_{+},R_{-})\leq |x^{+}-x^{-}|+|M^{+}-M^{-}|\leq 2q^{|I|+1}.
$$

For the first condition we actually need to prove only the case
$$
\dist(K^{-},R_{+})\geq10q^{2(|I|)+1},
$$
the other case follows directly from Corollary~\ref{cr}.
Choose a coordinate system such that $x^{-}>x^{+}=0$ and choose $c=(c_{1},c_{2})\in K^{-}$.
Denote $a=(x^{+},0)=(0,0),$ $b=(x^{+},M^{+})=(0,M^{+})$ and $(u,v)=(c_{1}-x^{-},c_{2}-M^{-})$.
Then we have $v\geq 0$, $|u|\leq D_{\alpha}q^{|I|+2}$ and the first condition transforms into
$$
\sqrt{(x^{-}+u)^{2}+(M^{-}+v)^{2}}-M^{+}\geq  q^{2|I|+2}.
$$
Using (\ref{mineq}) we obtain
$$
\begin{aligned}
\sqrt{(x^{-}+u)^{2}+(M^{-}+v)^{2}}-M^{+}\geq &\sqrt{(x^{-}+u)^{2}+(M^{-})^{2}}-M^{-}-|M^{+}-M^{-}|\\
\geq &\frac{(x^{-}+u)^{2}}{2M^{-}}- 10q^{2(|I|+1)}\\
\geq &\frac{3}{8q}((x^{-})^{2}-2x^{-}|u|)- 10q^{2(|I|+1)}\\
\geq &\frac{3}{8q}(q^{2(|I|+1)}-D_{\alpha}^{2}q^{2|I|+3})- 10q^{2(|I|+1)}\\
\geq &\frac{3}{8}q^{2|I|+1}-\bigl(10+\frac{3D_{\alpha}^{2}}{8}\bigr)q^{2(|I|+1)}\\
\geq &\frac{3}{8}q^{2|I|+1}-11q^{2(|I|+1)}\geq\frac{q^{2|I|+1}}{4}\geq q^{2|I|+2}.
\end{aligned}
$$
\end{proof}

The following proposition provides us with the inductive procedure that will allow us to find appropriate $\alpha.$
Geometrically, as we have seen in the proof of the previous lemma, we will be looking to obtain the vector of translation of two nearest copies of $K^{\alpha}$ to be almost parallel to $x$-axis.
This way we will be able to touch some such copies with balls centered on the $x$-axis that will have positive distance from the rest of the set.
Algebraically, what we are basically looking for is an $\alpha$ with a very specific Diophantine approximation.

\begin{proposition}\label{key}
There is a sequence of intervals $[\kappa_{n},\nu_{n}]\subset[0,2\pi]$ a sequence $k_{n}\in\en$ and a mapping $\Phi:\{\pm1\}^{<\omega}\mapsto\{2\}\times\{\pm1\}^{<\omega}$ with $\Phi(\emptyset)=\emptyset$
such that for every $n\in\en$
\begin{enumerate}\renewcommand{\labelenumi}{(\Alph{enumi})}
\item if $n>1$ then $k_{n}>2k_{n-1}+1$
\item if $n>1$ then $[\kappa_{n},\nu_{n}]\subset[\kappa_{n-1},\nu_{n-1}]$
\item for every $I\in\{\pm1\}^{n-1}$ and every $\alpha\in[\kappa_{n},\nu_{n}]$ we have 
$$
\dist(K^{\alpha}_{\Phi(I^*\pm1)}),R_{\Phi(I^*\pm1)})\geq q^{2k_{n}+2}
$$
\item for every $I\in\{\pm1\}^{n-1}$ and every $\alpha\in[\kappa_{n},\nu_{n}]$ we have 
$$
11q^{k_{n}+1}\geq|M^{\alpha}_{\Phi(I^*1)}-M^{\alpha}_{\Phi(I^*-1)}|\geq q^{k_{n}+1}
$$
\item $\Phi(I)\in\{2\}\times\{\pm1\}^{k_{n}}$
\item if $n>1$ and $I\in\{\pm1\}^{n-1}$ then $\Phi(I)\triangleleft\Phi(I^*{\pm1})$
\item if $I\in\{\pm1\}^{n-1}$ then 
\begin{equation}
M_{\Phi(I^*-1)}^{\alpha}\leq M_{J}^{\alpha}\label{mineq2}
\end{equation}
for every $J\in\{2\}\times\{\pm1\}^{k_{m}}$ such that $\Phi(I)\triangleleft J$ end every $\alpha\in[\kappa_{n},\nu_{n}]$
\item if $n>1$ and $I\in\{\pm1\}^{n-1}$ then 
$$
11q^{k_{n}+1}\geq M_{\Phi(I^*\pm1)}^{\alpha}-M_{\Phi(I)}^{\alpha}\geq0
$$ 
for every $\alpha\in[\kappa_{n},\nu_{n}].$
\item if $n>1$ then for every $I\in\{\pm1\}^{n-1}$, $\alpha\in[\kappa_{n},\nu_{n}]$, $R\in\R_{\Phi(I)}$ and $R_{\pm}\in\R_{\Phi(I^*\pm1)}$ we have
$$
\dist(R_{\pm},R)\leq 2q^{k_{n}+1}
$$
\end{enumerate}
\end{proposition}

\begin{proof}
We will proceed by induction by $n.$

For $n=1$ we use Lemma~\ref{intervals} for $k_{0}=0$ and $[a,b]=[0,2\pi]$ to obtain
$[\kappa_{1},\nu_{1}]$ and $k_{l}$ such that $(a)$, $(b)$ and $(c)$ hold with $k=k_{1}$ and $[c,d]=[\tilde\kappa_{1},\tilde\nu_{1}].$
Put $\Phi(\emptyset)=\emptyset$ and choose $\Phi(-1)\in\{2\}\times\{\pm1\}^{k_{1}}$ and 
$[\kappa_{1},\nu_{1}]\subset[\tilde\kappa_{1},\tilde\nu_{1}]$ in such a way that (\ref{mineq2}) holds and so we have property $(G)$.
Properties $(C)$, $(D)$ and $(I)$ follows from $(a)$, $(b)$ and $(c)$ in Lemma~\ref{intervals}, respectively.

{\it Induction step.} Choose $l\in \en$ and suppose that conditions $(A)$-$(I)$ hold for every $n<l.$
First use Lemma~\ref{intervals} to find $[\tilde\kappa_{l},\tilde\nu_{l}]\subset[\kappa_{l-1},\nu_{l-1}]$ and $k_{l}>2k_{l-1}+1$ 
such that $(a)$, $(b)$ and $(c)$ hold with $k=k_{l}$ and $[c,d]=[\tilde\kappa_{l},\tilde\nu_{l}].$
Now, find an interval $[\kappa_{l},\nu_{l}]$ and $J\in\{2\}\times\{\pm1\}^{k_{l}-1}$ with $\Phi(I)\triangleleft J$ 
such that (\ref{mineq2}) holds for every $\alpha\in[\kappa_{l},\nu_{l}]$ after we put $\Phi(I^*-1)=J^*-1.$
Put $\Phi(I^*1)=J^*1$.
Validity of all conditions $(A),(B)$, $(E)$-$(G)$ and the second inequality in condition $(H)$ follows directly from the above construction
and conditions $(C)$ and $(D)$ follow from $(a)$ and $(b)$ in Lemma~\ref{intervals}, respectively.
To prove the first inequality in $(H)$ it is sufficient to use the fact that $M_{\Phi(I^*-1)}^{\alpha}=M_{\Phi(I)}^{\alpha}$ and condition $(D).$
Condition $(I)$ holds due to the property $(c)$ in Lemma~\ref{intervals} and since in fact for every $R\in\R_{\Phi(I)}$ we have $R\in R_{\Phi(I^*-1)}.$
\end{proof}

\begin{theorem}
There is $\alpha\in [0,2\pi]$ such that $K^{\alpha}$ has uncountably many critical values.
\end{theorem}
\begin{proof}
Let $\Phi$ and $[\kappa_{n},\nu_{n}]$ be as in Proposition~\ref{key}.
Due to the property $(B)$ we can choose $\alpha\in\cap_{n\in\en}[\kappa_{n},\nu_{n}].$
We will prove that $K^{\alpha}$ has uncountably many critical values.
For $J\in\{-1,1\}^{\omega}$ put $M_{J}=\lim_{n\to\infty}M^{\alpha}_{\Phi(J|_{n})}.$
The limit exists, because $M^{\alpha}_{\Phi(J|_{n})}$ is a monotone sequence due to property $(H)$ (and obviously bounded).
We will prove that $M=\{M_{j}:J\in\{-1,1\}^{\omega}\}$ is uncountable.

First we prove that if $J\not=\tilde J$ then $M_{J}\not=M_{\tilde J}.$
Choose such $J$ and $\tilde J$ and let $n$ be the lowest number satisfying $J(n)\not=\tilde J(n).$
We can suppose that $-1=J(n)=-\tilde J(n).$
By properties $(A)$ and $(D)$ we have 
$$
\begin{aligned}
|M_{J}-M_{\tilde J}|\geq& |M^{\alpha}_{\Phi(J|_{n})}-M^{\alpha}_{\Phi(\tilde J|_n)}|
-\sum _{k=n}^{\infty} |M^{\alpha}_{\Phi(J|_{k})}-M^{\alpha}_{\Phi(J|_{k+1})}|\\
&-\sum _{k=n}^{\infty} |M^{\alpha}_{\Phi(\tilde J|_{k})}-M^{\alpha}_{\Phi(\tilde J|_{k+1})}|
\geq q^{2k_{n}+1}-22\sum _{m=2k_{n}+2}q^{m} >0.
\end{aligned}
$$

To finish the proof of the theorem, it is sufficient to prove that every $M_{J}$ is a critical value of $K^{\alpha}.$ 
Define 
$$
(u_{1},u_{2})=u_{J}=\bigcap_{n=0}^{\infty} K^{\alpha}_{\Phi(J|_{n})}
$$
(the above intersection is nonempty by property $(F)$) and put $S:=B((u_{1},0),u_{2}).$
Note that by property $(E)$ we have $u_{J}\in K_{2}^{\alpha}.$
Since $\dist((u_{1},0),K_{1}^{\alpha}\cup K_{-1}^{\alpha})\geq\frac{1}{3}$ and $u_{2}\leq q$ and using the fact that $K_{2}^{\alpha}$ 
and $K_{-2}^{\alpha}$ are symmetrical with respect to the $x$-axis we are done if we can prove $S\cap K_{2}^{\alpha}=\{u_{J}\}.$
This way we obtain $S\cap K=\{(u_{1},u_{2}),(u_{1},-u_{2})\}$ and it suffices to use obvious fact that $(u_{1},0)\in\co\{(u_{1},u_{2}),(u_{1},-u_{2})\}.$

So choose $z\in K_{2}^{\alpha}\setminus\{u_{J}\}$.
Then $z=\cap _{n}K^{\alpha}_{G|n}$ for some $G\in\{2\}\times\{-1,1\}^{\omega}.$
Due to properties $(E)$ and $(F)$ we can find $V\in\{2\}\times\{\pm1,\pm2\}^{\omega}$ such that $V|_{k_{n}+1}=\Phi(J|_n)$ for every $n\in\en.$
Find first $k$ such that $G(k)\not=V(k).$
Then there is an $m$ with the property that $k_{m}\geq k-1> k_{m-1}.$
Choose $R_{n}\in \R^{\alpha}_{\Phi(J|_{m+n})}$ arbitrary for every $n\in\en_0$.
By property $(I)$ we have for $n\in\en_0$
$$
R_{n+1}\subset R_{n}+B(0,2q^{k_{m+n}+1}).
$$
Since $R_{n}\to S$ in the Hausdorff metric we can write
$$
\begin{aligned}
S&\subset R_{0}+B(0,2\sum_{n=1}^{\infty}q^{k_{m+n}+1})\\
&\subset R_{0}+B(0,2\sum_{i=2k_{m}+3}^{\infty}q^{i})\\
&\subset R_{0}+B(0,3q^{2k_{m}+3}),
\end{aligned}
$$
where the second inclusion holds due to property $(A)$.
On the other hand, if $k-1=k_{m}$ then by property $(C)$ we have $\dist(z,R_{0})\geq q^{2k_{m}+2}>3q^{2k_{m}+3}$ and if $k-1<k_{m}$ we have by property $(G)$ and Corollary~\ref{cr} that $\dist(z,R_{0})\geq \frac{q^{2k_{m}+1}}{7}>3q^{2k_{m}+3}$.
Therefore $z\not\in S.$
\end{proof}

\textbf{Acknowledgement} I would like to thank Jan Rataj for many helpful discussions on the topic and especially for the idea of Proposition~\ref{prop}.

\end{document}